\documentclass[article]{amsart}
\usepackage[utf8]{inputenc}    
 
\usepackage[french,english]{babel}     
 
\usepackage[cyr]{aeguill}      
\usepackage{amsmath, amsthm, amsfonts,stmaryrd }
\usepackage{mathrsfs}                      
\usepackage{mathtools}                     
\usepackage{booktabs}
\usepackage{todonotes}
\usepackage[all]{xy}
\usepackage[utf8]{inputenc}
\usepackage{mathtools}                     
\usepackage{todonotes}
\usepackage{ marvosym }
\usepackage{bbm}                         
\usepackage{framed}
\usepackage{graphicx}   

\usepackage{hyperref} 
\hypersetup{
    colorlinks=true,
    linkcolor=blue,  
    urlcolor=cyan,
    pdftitle={RiemannianSplines},
    pdfauthor={R. TAHRAOUI, F.X. VIALARD},
} 
\usepackage{theoremref}

\numberwithin{equation}{section}
\theoremstyle{plain}
\newtheorem{remark}{Remark}
\newtheorem{theorem}{Theorem}
\newtheorem{lemma}[theorem]{Lemma}
\newtheorem{corollary}[theorem]{Corollary}
\newtheorem{proposition}[theorem]{Proposition}

\theoremstyle{definition}
\newtheorem{definition}{Definition}

\def\ep{\varepsilon}

\def\x{\times}

\def\R{{\mathbb R}}
\def\N{{\mathbb N}}

\let\on=\operatorname

\newcommand{\ud}{\,\mathrm{d}}

\def\Inc=\operatorname{Inc}

\newcommand{\eqdef}{\ensuremath{\stackrel{\mbox{\upshape\tiny def.}}{=}}}

\begin{document}
\title{Metric completion of $\on{Diff}([0,1])$ with the $H^1$ right-invariant metric}
\author{S. Di Marino}
\address{INdAM}
\email{simone.dimarino@altamatematica.it}
\author{A. Natale}
\address{INRIA, Project team Mokaplan}
\email{andrea.natale@inria.fr}
\author{R. Tahraoui}
\author{F.-X. Vialard}
\address{Universit\'e Paris-Est Marne-la-Vallée, LIGM, UMR CNRS 8049}
\email{fxvialard@normalesup.org}
\date{\today}
\maketitle

\begin{abstract}
We consider the group of smooth increasing diffeomorphisms $\on{Diff}$ on the unit interval endowed with the right-invariant $H^1$ metric. We compute the metric completion of this space which appears to be the space of increasing maps of the unit interval with boundary conditions at $0$ and $1$. We compute the lower-semicontinuous envelope associated with the length minimizing geodesic variational problem. We discuss the Eulerian and Lagrangian formulation of this relaxation and we show that smooth solutions of the EPDiff equation are length minimizing for short times. \end{abstract}

\section{Introduction}
On the group of smooth diffeomorphisms $\on{Diff}([0,1])$ with boundary conditions $\varphi(0) = 0$ and $\varphi(1)=1$, we consider the metric induced by the $H^1$ right-invariant metric on this group. Between two smooth diffeomorphisms $\varphi_0,\varphi_1 \in \on{Diff}([0,1])$, the right-invariant  distance $\on{dist}$ is defined by
\begin{equation}\label{H1Norm}
\on{dist}(\varphi_0,\varphi_1)^2 = \inf_{v \in C^1([0,1]^2)} \int_0^1 \int_{M} v(t,x)^2 +\frac 14 (\partial_x v(t,x))^2 \ud x \ud t \,,
\end{equation}
where $v$ is a time dependent vector field on $[0,1]$ with $v(t,0) = v(t,1) = 0$ for all $t \in [0,1]$ and
under the flow equation constraint 
\begin{equation}\label{EqFlowEquation}
\partial_t \varphi(t,x) = v(t,\varphi(t,x))
\end{equation}
and time boundary conditions $\varphi(1) = \varphi_1$ and $\varphi(0) = \varphi_0$.
It has been proven in \cite{Michor2005} that this distance is not degenerate on the group of diffeomorphisms, contrary to the right-invariant $L^2$ metric on the group. The Euler-Lagrange equation is known as the Camassa-Holm equation \cite{Holm1993}.
\par
For this choice of parameters, the Camassa-Holm equation takes the form
\begin{equation}\label{EqCHEulerian}
\partial_t v - \frac 14\partial_{txx} v + 3\partial_x v v - \frac 12\partial_{xx}v \partial_x v - \frac 14\partial_{xxx} v v = 0\,.
\end{equation}
It is also possible to write this action in Lagrangian coordinates:
\begin{equation}\label{Lagrangian}
L(\varphi) = \inf \int_0^1 \int_{M} (\partial_t \varphi)^2 \partial_x \varphi + \frac{1}{4}\frac{(\partial_{tx}\varphi)^2}{\partial_x \varphi} \ud x\ud t 
\end{equation}
with the corresponding time boundary conditions. Note that the second term can be extended as a convex functional on time dependent measures.
We can write the Euler-Lagrange equation associated with the Lagrangian $L$
\begin{equation}
\begin{cases}
-\frac{\ud }{\ud t} \frac{\partial L}{\partial \dot{\varphi}} + \frac{\partial L}{\partial \varphi}= 0\\
\frac{\partial L}{\partial \dot{\varphi}} = 2 \dot{\varphi} \partial_x \varphi - \frac{1}{2} \partial_x \left( \frac{\partial_x \dot{\varphi}} {\partial_x \varphi}\right) = 2\dot{\varphi} \partial_x \varphi - \frac 12 \partial_{xt}\log(\partial_x \varphi)\\
\frac{\partial L}{\partial \varphi} = - \partial_x\left( \dot{\varphi}^2\right) + \frac{\partial_{x}}{4} (\partial_t\log(\partial_x \varphi))^2\,,
\end{cases}
\end{equation}
which gives the following equation
\begin{equation}
-2\partial_{tt}\varphi \partial_x \varphi - 2 \partial_t{\varphi} \partial_{xt}\varphi + \frac12 \partial_{xtt} \log(\partial_x \varphi)  - \partial_x\left( \dot{\varphi}^2\right) + \frac{\partial_{xt}}{4} \log(\partial_x \varphi) = 0\,.
\end{equation}
\par 
It is shown in \cite{GALLOUET2017} that smooth solutions of the Camassa-Holm equation on the circle $S_1$ \eqref{EqCHEulerian} are length minimizing for short times. Interestingly, local existence of smooth solutions to the Cauchy problem \eqref{EqCHEulerian} together with \eqref{EqFlowEquation} are guaranteed since there is no loss of regularity of the geodesic flow as proven in \cite{Constantin2003}. A simple adaptation of the proof suggests that this result should also hold on the unit interval taking into account the boundary conditions. After this short time, which can be quantitatively estimated in terms of the Hessian of the flow map, existence of minimizers is not known. To the best of our knowledge, state of the art results in proving existence of length minimizing curves on a group of diffeomorphisms with a right-invariant metric is contained in \cite{CompletenessDiffeomorphismGroup} where \emph{strong} (see \cite{Ebin1970} for more details on strong and weak metrics) Riemannian Sobolev metrics, above the $C^1$ critical index, on the group are studied. In \cite{CompletenessDiffeomorphismGroup}, due to the Sobolev embedding theorem, the standard method of calculus of variation has been applied and it has led to a Hopf-Rinow type of result on the group of diffemorphisms with strong Sobolev metrics. In the one dimensional case, it gives for instance that the usual Sobolev metric $H^{n}(S_1)$ for $n > 3/2$ is a complete Riemannian manifold such that between any two points there exists a length minimizing geodesic and the geodesic flow is globally well-posed.
In \cite{Andrea2018}, a relaxation à la Brenier of the length minimizing geodesics problem is studied but it can be proven not tight in dimension $1$ and this relaxation is possibly tight in greater dimensions, which is still an open question in \cite{Andrea2018}.
\par
The question we want to address hereafter is the computation of a tight relaxation of the functional \eqref{H1Norm} in the case of $M = [0,1]$ and the completion of the group of diffeomorphisms as a metric space. In comparison with \cite{CompletenessDiffeomorphismGroup}, the $H^1$ metric is a \emph{weak} Riemannian metric and the Sobolev embedding does not apply, nor a theorem such as Aubin-Lions-Simon's, which makes the variational study more subtle. The closest technical advances we could rely on is, to the best of our knowledge, the line of research opened by Di-Perna and Lions, such as \cite[Corollary 2.6]{CrippaEstimates} which shows estimates on the (integral) variation of the logarithm of the Lagrangian flow map to deduce compactness of the flow. Note that the space of vector fields that are in $L^2([0,1],H^1_0([0,1]))$ does not insure a well-defined flow, even in a weak sense using Di-Perna Lions or the more recent work of Crippa and Ambrosio. Indeed, the divergence of the vector field, in this one dimensional case, its first derivative, is not bounded in $L^1([0,1],L^\infty([0,1]))$. Actually, the fact that the compressibility of the flow is not bounded is an important feature of the solutions to the Camassa-Holm equation, it is well-known that there exist vector fields that describe a peakon-antipeakon (two particles $x_0 < x_1$
) collision, i.e. such that the Lagrangian flow map (see definition \eqref{ThLagrangianFlowMap}) is such that $\varphi(t,x_0) = \varphi(t,x_1)$ for a certain finite time $t>0$ and $x_0 < x_1$.
\par 
\emph{Strategy for the relaxed formulation and the metric completion: } We aim at finding the semi-continuous envelope of the functional in Lagrangian coordinates \eqref{Lagrangian} or in Eulerian coordinates \eqref{H1Norm}. Although the two formulations are equivalent in a smooth setting, they may differ on non-smooth maps. In order to show existence of minimizers using the formulation \eqref{H1Norm}, we prove that the flow constraint is stable with respect to weak convergence, based on Helly's selection theorem. 
\par
Our approach also enables to compute the metric completion of the group of diffeomorphisms endowed with the right-invariant $H^1$ metric. In fact, we show that the formulation \eqref{H1Norm} on $M = [0,1]$ can be extended to the space of non-decreasing functions $f$ of $[0,1]$ into $[0,1]$ such that $f(0) = 0$ and $f(1) = 1$, space that we denote by $\on{Mon}_+$. 
Our main result is the following theorem.
\begin{theorem}\label{ThMainTheorem}
For $\varphi_0,\varphi_1 \in \on{Mon}_+$, the functional $L$ on the space of time dependent vector fields $v \in L^2([0,1],H^1_0([0,1]))$, i.e. $v(t,0) = v(t,1) = 0$
\begin{equation}\label{EqNormDefinition}
\mathcal{L}(v) = \int_0^1 \int_{M} v(t,x)^2 +\frac 14 (\partial_x v(t,x))^2 \ud x \ud t\,,
\end{equation}
under the flow equation constraint \eqref{EqFlow} and time boundary conditions $\varphi(1) = \varphi_1$ and $\varphi(0) = \varphi_0$ admits minimizers. 
Moreover, denoting $d^2(\varphi_0,\varphi_1)$ the minimum value of $L$, $d$ defines a right-invariant distance on the space of non decreasing functions $\on{Mon}_+$ for which it is a complete metric space. 
\end{theorem}
Last, we prove a gamma convergence result which shows that for given sufficiently regular $\varphi_0,\varphi_1 \in \on{Diff}$, minimizing on (corresponding) regular paths gives the same infimum value than in Theorem \ref{ThMainTheorem}.

\section{Flow stability}

First, we start with the definition of the flow corresponding to a vector field in $L^2([0,1],H^1_0)$, since it has, in general, not a unique Lagrangian solution.
\begin{definition}\label{ThLagrangianFlowMap}
Let $v \in L^1([0,1],C([0,1]))$. Then $\varphi$ is said to be a Lagrangian flow associated with $v$ if 
\begin{itemize}
\item $x \mapsto \varphi(t,x)$ is increasing for every $t$;
\item for every $x$, the map $t \mapsto \varphi(t,x)$ is absolutely continuous and
\begin{equation}\label{EqFlow} \varphi(t,x)-\varphi(s,x) = \int_s^t v(r,\varphi(r,x)) \, d r \qquad \forall \, 0\leq s <t\leq 1.\end{equation}
\end{itemize}
\end{definition}

Importantly, a Lagrangian flow need not be unique and we will often use this property in the rest of the paper, see for instance Lemma \ref{ThIdentityToStep}. The following result shows that every two maps are connected through a Lagrangian flow, which is an equivalence relation.

\begin{proposition}[Equivalence relation]\label{ThTransitivity}
Let $\varphi_0,\varphi_1 \in \on{Mon}_+$ be two increasing maps. The relation defined on $\on{Mon}_+$ by
\par
\noindent
"there exists a Lagrangian flow such that $\varphi(0,x) =  \varphi_0(x)$ and $\varphi(1,x) =  \varphi_1(x)$"
\par
\noindent
is symmetric and transitive. Moreover, there always exists a Lagrangian flow between two increasing maps.
\end{proposition}
\begin{proof}
The symmetry is obtained just by time reversion of the Lagrangian flow, i.e. considering $\varphi(1-t,x)$. The transitivity follows by concatenation of Lagrangian flows. Last, the existence result follows from the next lemma \ref{ThIdentityToStep} which shows that every increasing map is connected to a particular increasing map. Thus, the equivalence class is the full set $\on{Mon}_+$.
\end{proof}

\begin{lemma}\label{ThIdentityToStep}
Let $\varphi_0 \in \on{Mon}_+$ be an increasing map. There exists a vector field $v \in L^2([0,1],H^1_0([0,1]))$ such that its Lagrangian flow satisfies $\varphi(0,x) =  \varphi_0(x)$ and $\varphi(1,x) = 1/2$ if $0<x<1$, $\varphi(1,1) = 1$ and $\varphi(1,0) = 0$ otherwise.
\end{lemma}

\begin{proof}
First, observe that the solutions of the real valued ODE $\dot{x} = x^{1/3}$ for $x(0)>0$ has a unique solution that can be written $x(t) = (x(0)^{2/3} + t)^{3/2}> t^{3/2}$.
The main point of the proof consists in using this vector field which is not Lipschitz, in order to use the nonuniqueness of solutions of the flow. Consider the autonomous vector field on $[0,1]$ defined on a neighborhood of $1$ by $x \mapsto -|x- 1|^{1/3}$ and extended on the rest of the interval by a smooth vector field vanishing at $x = 0$. Then, there is a unique solution to the flow equation \eqref{EqFlow} on $[0,1[$. At $x = 1$, we consider the path $\varphi(t,1) = 1$ which is solution to the Lagrangian flow equation. Now, remark that for every $x$ in the neighborhood of $1$, one has $\varphi(t,x) \leq 1-t^{3/2}$ at least for short times. It implies that for any $\varepsilon >0$ sufficiently small, $\lim_{x \to 1} \varphi(\varepsilon,x) \leq 1-a <1$ where $a>0$ is sufficiently small. Thus for $t >0$ sufficiently small, the set $\varphi(t,[0,1[)$ is strictly separated from $\varphi(t,1) = 1$. Using a similar vector field at $0$, we get that $\lim_{x \to 1} \varphi(\varepsilon,x) \geq a > 0$.
\par
Using the autonomous vector field on $[0,1]$ defined by $x \mapsto \on{sgn}(1/2-x)|x- 1/2|^{1/3}$. The associated Lagrangian flow we consider moves every point $x \in [a,1-a]$ and $x \neq 1/2$ goes to $1/2$ in finite time and stay fixed at $1/2$ after that time. The point $1/2$ is left fixed.
By composition of Lagrangian maps, we obtain the result.
\end{proof}

\par
We have the following stability result.

\begin{proposition} \label{EqFlowStability}
Let $\varphi_n$ be a Lagrangian flow associated with the vector field $v_n \in L^2([0,T]; H^1_0([0,1]))$, such that $\varphi_n(t,0)=0$ and $\varphi_n(t,1)=1$. Suppose that $v_n \rightharpoonup v$, there then exists a subsequence $\varphi_n \to \varphi$ converging pointwise and such that $\varphi$ is a Lagrangian flow for the vector field $v$. 
\end{proposition}

\begin{proof} First we observe that $v_n(t,1)=v_n(t,0)=0$; in particular we have 
$$\sup_{x \in [0,1]}  |v_n|(t,x) \leq \int_0^1 |\partial_x v_n|(t,x) \, dx \leq \left( \int_0^1 | \partial_x v_n|^2(t,x) \, dx \right)^{\frac 12}.$$
Then, we obtain the estimate for $s\leq t$,
$$ |\varphi_n(t,x) - \varphi_n(s,x)| \leq \int_s^t  \sup_{y \in [0,1]} |v_n|(t,y) \, d t  \leq \sqrt{t-s}  \left( \int_s^t \int_0^1 | \partial_x v_n|^2(r,x) \, dx \, dr \right)^{\frac 12}.$$

Since $\partial_x v_n \rightharpoonup \partial_x v$ in $L^2([0,1]^2)$ we have that $\| \partial_x v_n \|_{L^2([0,1]^2)}$ is equibounded; in particular we have that $\varphi_n$ are equi-H\"{o}lder in the spatial variable:
\begin{equation}\label{EqHolder} |\varphi_n(t,x) - \varphi_n(s,x)| \leq C \sqrt{t-s} \qquad \forall \, 0 \leq s < t \leq 1 \text{ and } \forall x \in [0,1].
\end{equation}
\par
We can now use the Helly selection theorem on a countable dense set $\{ t_i \} \subseteq [0,1]$, in order to get $\varphi_n(t_i,x) \to f_i(x)$ for every $x$ and every $i$. Using then \eqref{EqHolder} we obtain that there exists a unique $\varphi(t,x)$, which is again H\"{o}lder-continuous in the time variable (and uniformly in the space variable), such that $\varphi(t_i,x)=f_i(x)$ and moreover $\varphi_n(t,x) \to \varphi(t,x)$ for \emph{every} $(t,x) \in [0,1]^2$. Moreover, fixing $x \in [0,1]$, we also have that $t \mapsto \varphi_n(t,x)$ converges uniformly to $t \mapsto \varphi(t,x)$.
\par
Then, we use an equivalent definition for \eqref{EqFlow}:
 $$\varphi_n(t,x)-\varphi_n(s,x) = \int_s^t \int_0^{\varphi_n(r,x)} \partial_x v_n(r,y)  \, dy \, d r =  \int_s^t \int_0^{1} \partial_x v_n(r,y) \chi_{\varphi_n, x} \, dy \, d r,  $$
 where $\chi_{\psi,x} (r,y)= 1$ if $y \leq \psi(r,x)$ and $0$ otherwise. From  $\varphi_n(\cdot,x) \to \varphi(\cdot,x)$ uniformly and the boundedness of $\chi_{\varphi_n, x}$, we deduce $\chi_{\varphi_n, x} \to \chi_{\varphi, x}$ strongly in $L^2$. Using then the weak convergence in $L^2$ of $\partial_x v_n$ to $\partial_x v$, we can pass to the limit, obtaining:
  $$\varphi(t,x)-\varphi(s,x) =  \int_s^t \int_0^{1} \partial_x v(r,y) \chi_{\varphi, x} \, dy \, d r =\int_s^t v(r,\varphi(r,x)) \, d r,  $$
thus concluding the proof.
 \end{proof}

In order to prove the gamma convergence result, we need some results on the structure of the Lagrangian flow. We first prove that the discontinuities are fixed w.r.t. the time.

\begin{lemma}\label{ThCountableDiscontinuities}
Let $\varphi$ be a Lagrangian flow for $v \in L^2([0,1],H^1_0([0,1]))$, then there exists a countable set $(x_i)_{i \in I} \subset [0,1]$ which contains the discontinuity set (or jump set) of $x \mapsto \varphi(t,x)$ for all time $t \in [0,1]$.
\par
In other words, the Lagrangian flow can be decomposed in a pure jump part and a continuous part,
$\varphi(t,x) = \varphi_c(t,x) + \sum_i \delta_i(t) \mathbf{1}_{x\geq x_i}$ where $\varphi_c \in C^0(D)$, with $\delta_i(t)$ nonnegative functions.
\end{lemma}
\begin{proof}
Denote by $\on{Disc}(\psi)$ the set of discontinuity points of $\psi$ a nondecreasing map on $[0,1]$; it is at most countable.
\\
Since the flow is uniformly H\"older in time, for any $(t,y) \in [0,1]^2$ such that $\varphi(t,x_+) \neq \varphi(t,x_-)$, there exists an open neighborhood $O(t)$ of $t$ on which $y$ is a discontinuity point for $\varphi(t',\cdot)$ for every $t'\in O(t)$, and in particular, at a rational time. Then, the previous remark shows that $\cup_{t \in [0,1]} \on{Disc}(\varphi(t,\cdot)) \subset \cup_{t \in \mathbb{Q}} \on{Disc}(\varphi(t,\cdot))$ and the right-hand side is at most countable, which gives the result.
\end{proof}
We now show that every Lagrangian flow of a time dependent $H^1_0$-vector field can be approximated in $L^1(D)$ by a continuous flow associated with the same vector field.
\begin{proposition}[General Filling]\label{ThGeneralFilling}
Let $\varphi$ be a Lagrangian flow associated with $v \in L^2([0,1],H^1_0([0,1]))$, then for every $\varepsilon >0$, there exists $\varphi_\varepsilon$ Lagrangian flow still associated with $v \in L^2([0,1],H^1_0([0,1]))$ such that $\| \varphi_\varepsilon - \varphi \|_{L^1} \leq \varepsilon$ and $\varphi_\varepsilon$ is continuous on $D$.
\end{proposition}

\begin{proof}[Sketch of proof]
For the readability of the article, we give here the main arguments, see the proof \ref{prop:gfilling} in appendix for the details of the proof.
\par
Fix $\varepsilon >0$ a positive real number.
We use Lemma \ref{ThCountableDiscontinuities} to introduce the set of discontinuity of the Lagrangian flow. This jump set is at most countable, say $(x_i)_{i \in \N}$ so that we can choose a summable sequence of positive real numbers $\varepsilon_i$ such that $\sum_{i = 1}^\infty \varepsilon_i = \varepsilon$. Note that, for each $x_i$ there is a countable union of open time \emph{disjoint} intervals in $[0,1]$ such that $x_i$ is a discontinuity point of the flow. For each of these intervals indexed by $j$, choose a reference time $t_i^j$.
\par
Hereafter, we assume that $\varphi(t,x_i)$ is equal to either both left or right limit.
We define the positive Radon measure $\mu_\varepsilon \eqdef \sum_{i = 1}^\infty \varepsilon_i \delta_{x_i} +  \on{Leb}$. Then, the function $x \mapsto \mu_\varepsilon([0,x])$ is of bounded variations. Let us denote its inverse by $G_\varepsilon$ for a moment, although it is not well-defined at discontinuity points. Then, we define 
\begin{equation}\label{EqDefiningApproximatedFlow}
\varphi_\varepsilon(t,y) = \varphi(t,G_\varepsilon(y))
\end{equation}
which implies that $\varphi_\varepsilon$ is defined everywhere but not on $[\mu_\varepsilon([0,x_i])_-,\mu_\varepsilon([0,x_i])_+]$. On each of these intervals, we define $\varphi_\varepsilon$ to be the interpolation given by the flow of the minimal norm of the $H^1$ vector field that interpolates the boundary conditions $\partial_t \varphi(t,x_i)_- = v(t,\varphi(t,x_i)_-)$ and $\partial_t \varphi(t,x_i)_+ = v(t,\varphi(t,x_i)_+)$. This time dependent vector field can be integrated to give a flow that completely defines the map $\varphi_\varepsilon$. This vector field reproduces the minimal norm given by $v_\varphi$. However, in order to integrate the flow, we need to give the map $\varphi$ at a given time $t$ which interpolates between the two limits $[\varphi(t,x_i)_-,\varphi(t,x_i)_+]$ when they differ. This interpolation can be chosen arbitrarily for each time $t_i^j$.
\par
The more general case when the Lagrangian flow at a discontinuity point is not equal to its left or right limit can be addressed by introducing a measure $\mu_\varepsilon \eqdef \sum_{i = 1}^\infty (\varepsilon_i^+ + \varepsilon_i^-) \delta_{x_i} +  \on{Leb}$ which accounts for discontinuities on the left and on the right, namely $\varepsilon_i^-$ (resp. $\varepsilon_i^+$) takes care of the discontinuity $\varphi(t,x_i) - \varphi(t,x_i^-)$ (resp. $\varphi(t,x_i^+) - \varphi(t,x_i^-)$).
\par
Now, we reparametrize the (space) interval in order to satisfy the boundary conditions. We have constructed the approximation $\varphi_\varepsilon: [0,1] \times [0,1+ \varepsilon] \mapsto [0,1]$ and using the linear map $S_\varepsilon:x \mapsto x/(1+\varepsilon)$, one can defined $\tilde{\varphi}_\varepsilon(t,x) \eqdef \varphi_\varepsilon(t,S_\varepsilon(x)): [0,1]^2 \mapsto [0,1]$. Since the energy is completely defined on the vector field $v_\varepsilon$, it is left unchanged.
\par
As done in Lemma \ref{EqFlowStability}, Helly's selection theorem can be applied when $\varepsilon \to 0$, and thus, the sequence $\varphi_\varepsilon$ converges in $L^1(D)$ and using Formula \eqref{EqDefiningApproximatedFlow}, one concludes that its limit is $\varphi$ (in fact, its pointwise limit almost everywhere).
\end{proof}
Now, we are able to prove a change of variable formula, which follows from standard calculus in the smooth case, but which still holds in the framework of definition \ref{ThLagrangianFlowMap}.
\begin{lemma}
Let $\varphi$ be a Lagrangian flow associated with $v$ and let $\psi$ be any (generalized) inverse of $\varphi$ in the $x$ variable. Then, for every $C^2$ function $f$ on the domain $D$, it holds
\begin{equation}\label{EqWeakChangeOfVariable}
\langle \partial_x v,f \circ \psi \rangle = \langle \partial_{tx} \varphi , f \rangle\,.
\end{equation}
It defines $\partial_{tx} \varphi$ as a Radon measure.
\end{lemma}
 
\begin{proof}
This formula is satisfied for a continuous Lagrangian flow due to the change of variable formula \cite{ChangeOfVariable}. Now, consider a Lagrangian flow which may have discontinuities, then, by Proposition \ref{ThGeneralFilling}, one can approximate it in $L^1(D)$ with continuous flows denoted by $\varphi_n$. Importantly, the $L^1$ convergence of $\varphi_n \to \varphi$ implies $L^1$ convergence of $\psi_n \to \psi$ for every choice of generalized inverses, since on $D$ the graphs of $\varphi$ and $\psi$ are symmetric w.r.t. the diagonal. Thus, $\psi_n$ converges in $L^1$ and Formula \eqref{EqWeakChangeOfVariable} holds true when passing to the limit; the left-hand side strongly converges in $L^2$ and the right-hand side also converges by integration by part on $f$.
\end{proof}

\begin{proof}[Proof of Theorem \ref{ThMainTheorem}]
By Proposition \ref{ThTransitivity}, the optimization set is non-empty. That is, between any two increasing maps $\varphi_0,\varphi_1$ on $[0,1]$, it is possible to find a Lagrangian flow $\varphi$ such that $\varphi(0,x) = \varphi_0(x)$ and $\varphi(1,x) = \varphi_1(x)$ for all $x \in [0,1]$.
\par
The existence of minimizers is implied by the stability result on the flow in Proposition \ref{EqFlowStability}.
\par
The right-invariance of $d$ is given by the composition of the flow maps, as well as the triangle inequality. The nonnegativity of $d$ is obvious and the fact $d(\varphi_0,\varphi_1) = 0$ implies pointwise equality follows from Equation \eqref{EqDistanceDominateSup}.
\par
We now prove completeness.
We first remark that the right-invariant distance dominates the pointwise sup norm, defined by
\begin{equation}
\| f \|_{\infty} = \sup_{x \in [0,1]} |f(x)|\,,
\end{equation}
and note that it is \emph{not} an essential supremum. We have
\begin{equation}\label{EqDistanceDominateSup}
\| \varphi_0 - \varphi_1 \|_\infty \leq 2 d(\varphi_0,\varphi_1)\,,
\end{equation}
which comes\footnote{The multiplicative factor $2$ in front of the distance is due to the fact that there is a $1/4$ factor in \eqref{EqNormDefinition}.} from the direct estimation, by application of Cauchy-Schwarz inequality,
\begin{align*}
| \varphi(t,x) - \varphi(0,x) | &= \left|\int_0^t v(s,\varphi(s,x)) \ud s \right| \leq \int_0^t \| v \|_{\infty} \ud s\\
& \leq \int_0^t \| v \|_{H^1} \ud s \leq \left(\int_0^1 \| v \|_{H^1}^2 \ud s\right)^{1/2}\,.
\end{align*}
Consider now a Cauchy sequence for $n$ a positive integer, $\varphi_n \in \on{Diff}([0,1])$. By the remark above, this sequence induces a sequence which uniformly converges under the sup norm. Therefore, it defines a limit map $\varphi_\infty \in \on{Mon}_+$, which is still nondecreasing.
\par
To prove that $d(\varphi_n,\varphi_\infty) \to 0$, consider a subsequence (without relabeling) such that $d(\varphi_n,\varphi_{n+1}) \leq 1/2^n$ and denote by $v_n$ a minimizer of the energy $L$.
It is sufficient to concatenate in time the vector fields $\tilde{v}_n$ for $n \geq N$ which are the unit speed parametrization of the vector field $v_n$ (if it is not already the case) over a (time) segment of length $d(\varphi_n,\varphi_{n+1})$. Thus, the resulting vector field $V_N$ is defined on the time interval $[0,\sum_{n = N}^\infty d(\varphi_n,\varphi_{n+1})]$. Now, we are left with proving that the corresponding flow at time $\sum_{n = N}^\infty d(\varphi_n,\varphi_{n+1})$ is equal to $\varphi_\infty$, but it is the result of the identification of the limit above. Therefore, this construction gives the estimation $d(\varphi_n,\varphi_\infty) \leq \sum_{n = N}^\infty d(\varphi_n,\varphi_{n+1}) \leq \sum_{n=N}^\infty 1/2^{N} \to_{N \to \infty} 0$.
\end{proof}

\begin{remark}[Uniqueness]
Note that general arguments for establishing uniqueness, such as strict convexity, do not hold here since the optimization problem is not convex. In fact, on $M = S_1$, rotational symmetry probably implies the existence of distinct minimizing geodesics.
On $M = [0,1]$, the rotational symmetry is broken and might be sufficient, together with the one dimensional context, for proving uniqueness. 
\end{remark}

\section{Lagrangian formulation and gamma convergence}\label{SecGamma}
In this paragraph, we are interested in the link between the Eulerian formulation that is well suited for the direct method of calculus of variations developed above and a pure Lagrangian formulation. It is important to note that to a Lagrangian flow correspond many different vector fields, unless $\varphi$ is continuous, or equivalently surjective. If $\varphi$ has discontinuity points, by minimizing over associated vector fields, it is possible to rewrite the Eulerian energy in terms of the Lagrangian map, once the discontinuity locations are fixed.
 
We now introduce a new estimate on smooth paths. 
\begin{lemma} \label{ThEstimate1} 
For a smooth path $\varphi(t)$ of bounded energy,
the function $\partial_x (\partial_t \varphi)^2$ is bounded in $L^1(t \times x)$ by a constant which only depends on the energy and thus $(\partial_t \varphi)^2$ is in $L^1(t,\on{BV}_x)$. 
\end{lemma}

\begin{proof}
This is given by the inequality
\begin{equation}
| \partial_t \partial_x \varphi \, \partial_t \varphi | \leq \frac 12 (\partial_t \varphi)^2 \partial_x \varphi +  \frac{1}{2}\frac{(\partial_{tx}\varphi)^2}{\partial_x \varphi} \,.
\end{equation}
The l.h.s is indeed $|\partial_x (\partial_t \varphi)^2|$.
\end{proof}

\begin{lemma}\label{ThLagrangianEulerianEnergies}
If the flow map $\varphi$ associated with $v$ is continuous in both $(t,x)$ and such that $\partial_t \varphi(t,x)\in L^2([0,1]^2)$, then the velocity field $v$ is uniquely defined. Moreover, the Lagrangian \eqref{Lagrangian} functional is equal to the Eulerian  functional \eqref{H1Norm}.\end{lemma}

\begin{proof}
Since the image of $\varphi(t,\cdot)$ is equal to the whole interval $[0,1]$ is a priori overly determined since the following equation has to be satisfied $\partial_t \varphi(t, x) = v(t,\varphi(t,x))$, or in other words $\partial_t \varphi(t,x)$ should be constant on level set of $\varphi(t,x)$. This is a required property of elements in $\mathcal{F}$ and the condition is satisfied a.e.
\par
We also have
\begin{equation}\label{EqEqualityKineticEnergy}
\int_D | \partial_t \varphi \circ \varphi^{-1}|^2 \ud x = \int_D | \partial_t \varphi |^2 \varphi^{-1}_*(\ud x) = \int_D |\partial_t \varphi |^2 \partial_x \varphi\,,
\end{equation}
where the last equality has a well-defined meaning since $\partial_x (\partial_t \varphi)^2$ is a Radon measure by Lemma \ref{ThEstimate1} and $\varphi$ is continuous on the domain $D$.
Now, we prove a first inequality by estimating $\int_D \partial_x v \, f $ for $f \in C^1(D)$,
\begin{align}
\int_D  v (-\partial_x f) &=- \int_D  v \circ \varphi \, \partial_x f \circ \varphi \, \partial_x \varphi \label{Eqtoto}\\
& = -\int_D  v \circ \varphi \, \partial_x (f \circ \varphi) = \int_D \partial_{xt} \varphi f\circ \varphi \label{Eqtutu} \\
& \leq 4\on{FR}(\partial_{tx} \varphi,\partial_x \varphi)^{1/2}\left( \int_D f^2 \circ \varphi \, \partial_x \, \varphi \right)^{1/2}  = 4\on{FR}(\partial_{tx} \varphi,\partial_x \varphi)^{1/2}\| f \|_{L^2(D)} \label{Eqtata}
\end{align}
where $\on{FR}(\partial_{tx} \varphi,\partial_x \varphi) = \frac 14 \int_D \frac{(\partial_{tx} \varphi)^2}{\partial_x \varphi} \ud x$ is the functional which we call Fisher-Rao. In the previous inequalities, we used 
the chain rule for the composition between $C^1$ and $BV$ functions between \eqref{Eqtoto} and \eqref{Eqtutu} and also to obtain the last inequality. We also applied the Cauchy-Shwarz inequality to obtain inequality \eqref{Eqtata}, since 
\begin{equation}
 \int_0^1 \int_{M} \frac{(\partial_{tx}\varphi_\eta)^2}{\partial_x \varphi_\eta} \ud x\ud t  \leq  \int_0^1 \int_{M} \frac{(\partial_{tx}\varphi)^2}{\partial_x \varphi} \ud x\ud t\,,
\end{equation}
where $\varphi_\eta \eqdef \eta \star \varphi$ a regularization of $\varphi$ by a kernel $\eta$ and convexity of the functional.
Passing by, we have proven that $\partial_x v$ is in $L^2([0,1])$.
\par
The second inequality is more involved and requires the use of the definition of the Fisher-Rao functional as a Legendre transform. One has
\begin{equation}\label{EqDualDefinition}
\on{FR}(\nu,\mu) = \sup_{u,w \in C^0(D)} \left[ \int_D u \ud \mu + \int_D w \ud \nu - \int_D \iota_K(u,w)  \ud t \ud x \right]  \,,
\end{equation}
where $\iota_K$ is the indicator function of the convex set
\begin{equation*}
K \eqdef \left\{ (\xi_1,\xi_2) \in \R^2 : \xi_1 + \xi_2^2 \leq 0  \right\}\,.
\end{equation*}
Actually, $\iota_K$ is the Legendre-Fenchel conjugate of 
 $r:\R \times \R \to \R_+ \cup \{+ \infty \}$ be the one-homogeneous convex function defined by
\begin{equation}
r(x,y) = 
\begin{cases}
\frac 14 \frac{y^2}{x} \text{ if } x > 0 \\
0 \text{ if } (x,y) = (0,0)\\
+\infty \text{ otherwise.}
\end{cases}
\end{equation}
Consider now $\varepsilon >0$ and a couple $(u,w) \in C^0(D)$ such that, the r.h.s. of Formula \eqref{EqDualDefinition} is greater than $\on{FR}(\partial_{tx} \varphi,\partial_x \varphi) - \varepsilon$.
Then, we will consider  $u = -w^2$ (without loss of generality). We choose a test function $z = w \circ \varphi^{-1}$, we have 
\begin{align*}
\langle \partial_x v,z \rangle &= \langle \partial_x v , w \circ \varphi^{-1}\rangle \\
& =\langle \partial_{x} v \circ \varphi \, \partial_x \varphi, w \rangle \\
& =\langle \partial_{tx} \varphi , w \rangle \\
& \geq \on{FR}(\partial_{tx} \varphi,\partial_x \varphi) - \varepsilon + \langle \partial_x \varphi, w^2 \rangle \\
& \geq \on{FR}(\partial_{tx} \varphi,\partial_x \varphi) - \varepsilon + \|z\|^2_{L^2(D)} \,.
\end{align*}
In particular, it implies that the polynomial function 
\begin{equation}
f(\lambda) = -\lambda \langle \partial_x v,z \rangle + \on{FR}(\partial_{tx} \varphi,\partial_x \varphi) - \varepsilon + \lambda^2 \|z\|^2_{L^2(D)}
\end{equation}
has a nonnegative discriminant, that is
\begin{equation}
| \langle \partial_x v,z \rangle |^2 \geq 4 (\on{FR}(\partial_{tx} \varphi,\partial_x \varphi) - \varepsilon) \| z\|^2_{L^2(D)}\,.
\end{equation}
These two inequalities and the first equality on the kinetic energy \eqref{EqEqualityKineticEnergy} give the claimed equality between the Lagrangian and Eulerian functionals.
\end{proof}

We now prove that if the initial and final diffeomorphisms are smooth enough, the minimization on vector fields that are in the same smoothness category gives the same minimization result than in Theorem \ref{ThMainTheorem}. The result is based on the right-invariance of the metric which enables the construction of smooth approximating sequences. Let us describe the underlying strategy developed in the proof below. Recall the result of Proposition \ref{ThGeneralFilling}: If a Lagrangian $\varphi(t,x)$ has a non-empty jump set, an approximation of it can be defined by introducing an interval at each jump point on which we will define a minimal norm interpolation. By right-invariance of the metric, this solution, which is defined on a larger interval than $[0,1]$, can be mapped to $[0,1]$ while preserving the total energy. This new solution provides a continuous path which is an approximation of the initial flow. Starting from this candidate, we use standard smoothing arguments, the main point consists in dealing with the boundary conditions.
\begin{theorem}\label{ThTightnessTheorem}
Let $\varphi_0,\varphi_1$ be two (different) $W^{1,1}$ (resp. $C^k$, $H^k$, $k \geq 1$) non decreasing functions on $[0,1]$ fixing the points $0,1$.
Let $\varphi$ be a Lagrangian flow associated with a vector field $v \in L^2([0,1],H^1)$, then, there exists a sequence of Lagrangian flows $\varphi_n$ converging in $L^1$ to $\varphi$ such that $\varphi_n$ is $W^{1,1}$ (resp. $C^k$) on $D$ and such that $\liminf L(\varphi_n) \leq \mathcal{L}(v)$. 
\end{theorem}

\begin{proof}
Proposition \ref{ThGeneralFilling} implies that it is sufficient to prove the result on a continuous Lagrangian flow. Therefore, we consider a continuous Lagrangian flow associated with a vector field $v \in L^2([0,1],H^1_0)$. 
\par
\emph{Regularization by convolution: }
Assume that the boundary conditions are $C^k$, or $H^k$ for $k\geq 1$.
By the steps above, we now have a curve of continuous maps $\varphi(t,x)$ and we aim at approximating it by a smooth curve by convolution with a smooth and compactly supported kernel $k_\eta$ where $\eta$ is the width parameter. Note that in particular the boundary conditions at $x=0$ and $x=1$ are not preserved for sufficiently small $\eta$. Therefore, we extend $\varphi(t,x)$ by $0$ if $x < 0$ and $1$ if $x > 1$. As a result, the support of $\psi_\eta = k_\eta \star \varphi$ is contained in $[-\eta,1+\eta]$. Let $c_{\eta}$ be the affine map that transforms $[0,1]$ in $[-\eta,1+\eta]$ and define $\varphi_\eta \eqdef (k_\eta \star \varphi) \circ c_\eta$, which is a nondecreasing map in $\on{Mon}_+$. We compute
\begin{equation}
L(\varphi_\eta) =  \int_0^1 \int_0^1 (\partial_t \varphi_\eta)^2 \partial_x \varphi_\eta + \frac{1}{4}\frac{(\partial_{tx}\varphi_\eta)^2}{\partial_x \varphi_\eta} \ud x\ud t \,.
\end{equation}
We first make the change of variable with $c_\eta$ to obtain
\begin{equation}
L(\varphi_\eta) =  \int_0^1 \int_{-\eta}^{1+\eta} (\partial_t \psi_\eta)^2 \partial_x \psi_\eta + \frac{1}{4}\frac{(\partial_{tx}\psi_\eta)^2}{\partial_x \psi_\eta} \ud x\ud t \,.
\end{equation}
The second term (Fisher-Rao) is convex in $(\partial_{tx} \psi_\eta,\partial_x \psi_\eta)$ and as a consequence
\begin{equation}
 \int_0^1 \int_{-\eta}^{1+\eta} \frac{(\partial_{tx}\psi_\eta)^2}{\partial_x \psi_\eta} \ud x\ud t  \leq  \int_0^1 \int_{-\eta}^{1+\eta} \frac{(\partial_{tx}\varphi)^2}{\partial_x \varphi} \ud x\ud t\,.
\end{equation}
The first term is not convex but we can use convexity of the quadratic term to obtain
\begin{equation}
 \int_0^1  \int_{-\eta}^{1+\eta} (\partial_t \psi_\eta)^2 \partial_x \psi_\eta \ud x\ud t  \leq  \int_0^1  \int_{-\eta}^{1+\eta} k_\eta \star (\partial_t \varphi)^2 \partial_x \psi_\eta \ud x\ud t\,,
\end{equation}
we integrate by part to get
\begin{equation}
 \int_0^1  \int_{-\eta}^{1+\eta} -\partial_x (\partial_t \psi_\eta)^2  \psi_\eta \ud x\ud t  \leq  \int_0^1  \int_{-\eta}^{1+\eta} -\partial_x (\partial_t \varphi)^2  \psi_{2\eta} \ud x\ud t\,.
\end{equation}
Now, the second term converges to $\int_0^1 \int_0^1 -\partial_x (\partial_t \varphi)^2  \varphi \ud x\ud t$ since $\psi_{2\eta}$ uniformly converges to $\varphi$, due to its continuity. It implies that $ \lim_{\eta \to 0} L(\varphi_\eta)\leq L(\varphi)$. 
\par
\emph{Boundary conditions: }In order to finish the proof, we now take care of the boundary conditions at time $0$ and $1$. Recall that the energy $L$ represents the kinetic energy of a path in a space of maps, therefore it is possible to concatenate paths while the energy is subadditive (up to a positive multiplicative constant). We will prove that the evaluations at times $0,1$ of the maps $\varphi$, which are at least $W^{1,1}$, are close in the Hellinger distance on the jacobians. Therefore, we then conclude using Proposition \ref{ThmControlHellinger} and concatenation of paths, by noting that the interpolation of the square roots preserve regularity.
\\
We consider $\varphi_\eta(t=0)$ (the case $t = 1$ is similar) and we write 
\begin{multline}\label{EqFirstEstimate}
\| \partial_x \varphi_\eta - \partial_x \varphi \|_{L^1} \leq \| \partial_x c_\eta \partial_x \psi_\eta \circ c_\eta - \partial_x \psi_\eta \circ c_\eta\|_{L^1} \\+ \|\partial_x \psi_\eta \circ c_\eta - \partial_x \varphi \circ c_\eta \|_{L^1} + \| \partial_x \varphi \circ c_\eta - \partial_x \varphi \|_{L^1}\,.
\end{multline}
Since the space of Lipschitz functions is dense in $L^1$, there exists a Lipschitz function $f$ such that 
$\| f - \partial_x \varphi \|_{L^1} \leq \varepsilon$ and thus we have
\begin{align*}
\| \partial_x \varphi \circ c_\eta - \partial_x \varphi \|_{L^1} &\leq \| \partial_x \varphi \circ c_\eta - f\circ c_\eta \|_{L^1} +  \| f - f\circ c_\eta \|_{L^1} +  \| f - \partial_x \varphi \|_{L^1} \\
&\leq (1+2\eta)^{-1}\varepsilon + \on{Lip}(f)|\partial_x c_\eta - 1| + \varepsilon \,\\
&\leq (1+2\eta)^{-1}\varepsilon + 2\eta\on{Lip}(f) + \varepsilon \,.
\end{align*}
Using the previous estimate in \eqref{EqFirstEstimate}, we obtain
\begin{equation}
\| \partial_x \varphi_\eta - \partial_x \varphi \|_{L^1} \leq \frac{2\eta}{1 + 2\eta} + \|\partial_x \psi_\eta \circ c_\eta - \partial_x \varphi \circ c_\eta \|_{L^1} + (1+2\eta)^{-1}\varepsilon + 2\eta\on{Lip}(f) + \varepsilon\,.
\end{equation}
Since the second term converges to $0$ with $\eta$, the maps $\varphi_\eta$ and $\varphi$ are close in $W^{1,1}$ norm.
\end{proof}

\begin{proposition}\label{ThmControlHellinger} Let $\varphi_0, \varphi_1 \in W^{1,1}([0,1])$ be two non decreasing functions such that $\varphi_0(0)=\varphi_1(0)$ and $ \varphi_0(1)=\varphi_1(1)=1$. Then there exists $\varphi:[0,1]\times [0,1]$ such that $\varphi(0, \cdot)=\varphi_0$ and $\varphi(1, \cdot)=\varphi_1$  and moreover

$$\int_0^1 \int_0^1 (\partial_t \varphi)^2 \partial_x \varphi  + \frac{ (\partial_{tx} \varphi)^2}{\partial_x \varphi} \,dx \, dt \leq C \int_0^1 | \sqrt{\partial_x \varphi_1} - \sqrt{\partial_x \varphi_0}|^2 \, d x,$$
where $C$ is a universal constant.
\end{proposition}

\begin{proof} Let us define $d^2=\int_0^1 | \sqrt{\partial_x \varphi_1} - \sqrt{\partial_x \varphi_0}|^2 \, d x$ which is obviously different from $0$ if and only if the two increasing functions differ. Then let us consider $f_t(x)= t \sqrt{ \partial_x \varphi_1} + (1-t) \sqrt{ \partial_x \varphi_0}$; We define $\tilde{\varphi}(t,x)= \int_0^x f_t^2 (y)\, d y$. We have $\tilde{\varphi} (t,0)=0$ but $\tilde{\varphi}(t,1)=1-t(1-t) d^2$. Since we want $\varphi(t,1)=1$ also, we will define $\varphi(t,x)=\tilde{\varphi} (t,x) / (1-t(1-t)d^2)$. Then we have

$$ \sqrt{ \partial_x \varphi } =\frac{ \sqrt{ \partial_x \tilde{\varphi} }}   {\sqrt{1-t(1-t)d^2}} =  \frac { f_t }{\sqrt{1-t(1-t)d^2}}= \frac { t \sqrt{ \partial_x \varphi_1} + (1-t) \sqrt{ \partial_x \varphi_0} }{\sqrt{1-t(1-t)d^2}} $$

$$ \partial_t \sqrt{ \partial_x \varphi } = \frac {  \sqrt{ \partial_x \varphi_1} - \sqrt{ \partial_x \varphi_0} }{\sqrt{1-t(1-t)d^2}} - \sqrt{ \partial_x \varphi } \frac {(t-1/2)d^2}{1-t(1-t)d^2}  $$

$$\int_0^1 |\partial_t \sqrt{ \partial_x \varphi }|^2 \, dx  \leq  2 \frac { d^2 }{1-t(1-t)d^2}  + 2 \frac {(t-1/2)^2d^4}{(1-t(1-t)d^2)^2}.  $$

Now we can use that $d^2 \leq 2$ and $t(1-t) \leq 14$ to obtain 
$$\int_0^1 |\partial_t \sqrt{ \partial_x \varphi }|^2 \, dx \leq 4d^2 + 2d^4. $$

In a similar manner, we shall compute $\partial_t \varphi$. Let us start with computing $\partial_x\varphi$:

$$ \partial_x \varphi = \frac { \partial_x \varphi_0 + 2t ( \sqrt{ \partial_x \varphi_1} \sqrt{ \partial_x \varphi_0 } - \partial_x \varphi_0 ) + t^2 ( \sqrt{ \partial_x \varphi_1 } - \sqrt{ \partial_x \varphi_0} )^2 }{1-t(1-t)d^2} $$ 

\begin{multline}\varphi(t,x) = \int_0^x \partial_x \varphi \, dy\\= \frac { \varphi_0(x,t) + 2t 	\int_0^x( \sqrt{ \partial_x \varphi_1} \sqrt{ \partial_x \varphi_0 } - \partial_x \varphi_0 ) \, dy + t^2 \int_0^x ( \sqrt{ \partial_x \varphi_1 } - \sqrt{ \partial_x \varphi_0} )^2  \, dy}{1-t(1-t)d^2}\,.
\end{multline} 

Let us notice that $\varphi(t,x)$ has a very simple expression in $t$. In fact we have $\varphi(t,x)= \tilde{\varphi}(t,x) / (1-t(1-t)d^2)$ where $\tilde{\varphi}$ is a quadratic polynomial in $t$ for every $x$. Then we have

$$ \partial_t \tilde{\varphi}  = 2 \int_0^x \sqrt{ \partial_x \varphi_0} (\sqrt{ \partial_x \varphi_1 } - \sqrt{\partial_x \varphi_0} ) \, dy + 2t \int_0^x ( \sqrt{ \partial_x \varphi_1 } - \sqrt{ \partial_x \varphi_0} )^2  \, dy$$
\begin{align*} |\partial_t \tilde{\varphi}|  & \leq 2 \left(\int_0^x \partial_x \varphi_0 \,dy \right)^{1/2} \cdot \left( \int_0^x (\sqrt{ \partial_x \varphi_1 } - \sqrt{\partial_x \varphi_0} )^2 \, dy \right)^{1/2} + 2td^2 \\
& \leq 2 \sqrt{\varphi_0} d + 2td^2.
\end{align*}

Then we compute $\partial_t \varphi$:
$$\partial_t \varphi = \frac{\partial_t \tilde{\varphi}}{1-t(1-t)d^2}- \varphi \frac{(2t-1)d^2}{1-t(1-t)d^2}$$
\begin{align*}
| \partial_t \varphi| & \leq \frac{|\partial_t \tilde{\varphi}|}{1-t(1-t)d^2} + \varphi \frac{|2t-1|d^2}{1-t(1-t)d^2} \\
& \leq \frac{  2 \sqrt{\varphi_0} d + 2td^2 + \varphi |2t-1|d^2 }{1-t(1-t)d^2} \\
& \leq 4d + 6d^2.
\end{align*}

In the end we can conclude using that $d^2 \leq 2$\footnote{We can take $C=144$ for example.}:

$$\int_0^1 \int_0^1 (\partial_t \varphi)^2 \partial_x \varphi  + \frac{ (\partial_{tx} \varphi)^2}{\partial_x \varphi} \,dx \, dt \leq (4d+6d^2)^2 + 4d^2+2d^4 \leq C d^2.$$
\end{proof}

As a corollary of the relaxation theorem, it is possible to rewrite the functional only in terms of the Lagrangian map $\varphi$.
\begin{corollary}\label{ThCor}
The relaxation of the functional \eqref{Lagrangian} for initial and final conditions in $W^{1,1}$ is the following
\begin{multline}\label{EqExtension}
L(\varphi) =  \int_0^1 \int_{S_1} (\partial_t \varphi)^2 \partial_x \varphi^c + \frac{1}{4}\frac{(\partial_{tx}\varphi^c)^2}{\partial_x \varphi^c} \ud x\ud t \\+  \int_0^1\sum_{\on{Jumps}(\varphi)}  (\varphi(x_+) - \varphi(x_-))\left((\partial_t \varphi(x_-))^2 + (\partial_t \varphi(x_+))^2 \right) \on{coth}(\varphi(x_+) - \varphi(x_-)) \\- \frac{2 (\varphi(x_+) - \varphi(x_-))}{\on{sinh}(\varphi(x_+) - \varphi(x_-))} \partial_t \varphi(x_+) \partial_t \varphi(x_-)\ud t\,,
\end{multline}
where  $\varphi^c$ is the continuous part of $\varphi$.
\end{corollary}
\begin{proof}
The result follows by  computing explicitly the minimal $H^1$ norm interpolant $v_\varphi$ at the jump set. Then, we use the constructed map $\tilde{\varphi}_\varepsilon$ which has the same energy than $v_\varphi$ on which we apply lemma \ref{ThLagrangianEulerianEnergies} which proves that the Lagrangian and Eulerian functionals on $\tilde{\varphi}_\varepsilon$ coincide. Now, the continuous part of the energy also satisfies right-invariance by a diffeomorphism so that the energy of the continuous parts of $\tilde{\varphi}_\varepsilon$ and $\varphi$ are equal.
\end{proof}
\begin{remark}
When the total variation of jumps converge to $0$, then the functional converges to the continuous part.
\end{remark}

\section{Short time geodesics are length minimizing}

In this section, we discuss the links with the results from \cite[Theorem 6.4]{Andrea2018} or \cite[Theorem 23]{GALLOUET2017} stating that solutions of the geodesic equation for sufficiently smooth initial conditions are length minimizers. The functional \eqref{Lagrangian} can be rewritten as 
\begin{equation}
\tilde{L}(z) = \inf_{z(t,x)} \int_0^1 \int_{0}^1 \| \partial_t z \|^2 \ud x\ud t 
\end{equation}
where $z = \sqrt{\partial_x \varphi(t,x)} e^{i \varphi(t,x)}$. Such a reformulation and its relaxation has been used in \cite[Theorem 23]{GALLOUET2017} to prove that smooth geodesics are length minimizing for short times. The general idea consists in rewriting the minimization problem on $(\varphi,\lambda)$ with an additional variable $\lambda$ as $z(t,x) = \lambda(t,x)e^{i\varphi(t,x)}$ and to put the constraint $\lambda(t,x) = \sqrt{\partial_x \varphi(t,x)}$ on $[0,1]$. Importantly, this constraint can be rewritten as a pushforward constraint, $\varphi_*(\lambda^2) = 1$ which allows the definition of a relaxation functional on the space of measures on the cone $[0,1] \times \R_{+}$ as proposed in \cite{Andrea2018}. 
The relaxed variational problem can be proven (see below) to be a lower bound for our variational relaxation \eqref{EqExtension} since the monotonicity is preserved whereas it is not the case for measure valued solutions in \cite{Andrea2018}. In fact, using this construction, it is possible to construct minimizing paths between maps that are not in the connected component of identity. 
For instance, there exists a generalized minimizing geodesic between the identity map and the "hat" map $x \in S_1 \mapsto 2x  \text{ mod } 1$ which  is obviously not a diffeomorphism. 
\par
In this present work, we have computed in Section \ref{SecGamma} the tight relaxation of the minimization problem \eqref{H1Norm}. Therefore, if smooth geodesics of the CH equation are short time minimizers in the sense of generalized solutions in \cite{Andrea2018}, they are a fortiori minimizers for our tight relaxation. Indeed, every generalized Lagrangian flow \eqref{ThCor} can be described as a generalized solution of \cite{Andrea2018}. We first remark that every flow $\varphi(t,x)$ which is $W^{1,1}$ in space can be lifted as a measure on the set of paths on the cone. Setting
\begin{align}
\iota_\varphi: [0,1] &\mapsto C([0,1],[0,1] \times \R_+)\\
x &\mapsto (t \mapsto (\varphi(t,x),\partial_x \varphi(t,x)))
\end{align}
so that $[\iota_\varphi]_*(\on{Leb})$ gives the probability measure on the path space of the cone. In the case there are jumps developing on the Lagrangian map, this definition needs to be adapted. Let us state the following theorem which is a consequence of \cite[Corollary 6.5]{Andrea2018}.
\begin{theorem}
Let $\varphi$ be a smooth solution of the Camassa-Holm equation on the time interval $[0,T]$. Equation \eqref{EqCHEulerian} can be rewritten as 
\begin{equation}
\begin{cases}
\frac{1}{2} \partial_{tx} u + \frac 14 (\partial_x u)^2  + \frac 12 u \partial_{xx} u -  u^2 = -2p \\
\partial_{t} u + 2 \partial_x u u  = -\partial_x p \,.
\end{cases}
\end{equation}
which defines the so-called pressure $p: [0,T] \times [0,1] \mapsto \R$.
If the following operator norm bound is valid
\begin{equation}\label{EqInequality11}
T^2 \left| \begin{pmatrix} \partial_{xx}p & 2\partial_x p \\ 2\partial_x p & 2p \end{pmatrix} \right| \leq \pi^2\,,
\end{equation}
then the solution $\varphi$ is a minimizer for the relaxed formulation \eqref{EqExtension} on the time interval $[0,T]$. Moreover, if the inequality \eqref{EqInequality11} is strict, then it is the unique minimizer.
\end{theorem}

\begin{proof}
The core of the proof consists in lifting a flow defined in \eqref{EqExtension} as a probability measure on the set of paths on the cone which is the relaxation space constructed in \cite{Andrea2018} which has the same cost. To do so, we use a similar construction to Proposition \ref{ThGeneralFilling}. The set of jumps being countable, one can index by position and time the set of jumps: $(x_i,t_{ij})$ for $i,j \in \N$. Moreover, one can choose the times $t_{ij}$ in such a way that 
\begin{equation}
\sum_{i,j} \varphi(t_{ij},x_i)_+ - \varphi(t_{ij},x_i)_- < \infty\,,
\end{equation}
because for a given jump $x_i,t_{ij}$, one can decrease $t_{ij}$ such that the size of the jump is arbitrarily close to $0$.
Then, the flow of $v$ minimizing the action defined by the condition $\on{id}$ at time $t_{ij}$ is well defined on $[ \varphi(t_{ij},x_i)_- , \varphi(t_{ij},x_i)_+ ]$ and during a time interval $[t_{ij}^-,t_{ij}^+]$ timepoints such that the jump disappears. Outside the time interval $[t_{ij}^-,t_{ij}^+]$, we extend the paths defined on the cone by $0$. Therefore, one considers
\begin{equation}
[\iota_\varphi]_*(\on{Leb}) + \sum_{i,j}[\iota_\varphi]_*(\mathbf{1}_[ \varphi(t_{ij},x_i)_- , \varphi(t_{ij},x_i)_+] \on{Leb})\,.
\end{equation}
This measure is finite on the path space and satisfies the marginal constraints.
However, it is not a probability measure and it can be alleviated by  using \cite[Lemma 4.5]{Andrea2018} which gives the existence of a probability measure which still satisfies the marginal constraints and which has the same energy. 
Then, the proof is a consequence of \cite[Corollary 6.5]{Andrea2018}. For the equivalent formulation of the Camassa-Holm equation using the "pressure" term, we refer the reader to \cite[Appendix, Equation (A.5)]{AndreaEmbedding}. 
\end{proof}
Actually, the proof of the previous theorem implies,
\begin{corollary}
The tight relaxation \ref{ThCor} is contained in the relaxation à la Brenier developed in \cite{Andrea2018} since every generalized solution of \ref{ThCor} correspond to a least one generalized flow in \cite{Andrea2018}.
\end{corollary}

\section{Perspectives}
In this article, we computed the tight relaxation of the boundary value problem associated with geodesics for the $H^1$ right-invariant metric on the group of diffeomorphisms of the unit interval with boundary conditions. We have shown that the relaxation of the problem can be defined on the space of nondecreasing maps of the unit interval which is the metric completion of the smooth diffeomorphism group.
An interesting issue is the smoothness of optimal paths for smooth boundary conditions. It is a natural question to study the regularity property of the minimizers in terms of that of the boundary conditions. In particular, we conjecture that discontinuities in the minimizing geodesics do not appear.  If so, the situation would be very different from the two dimensional case where measure solutions appear even if the boundary conditions are smooth (see \cite{Andrea2018}). Numerical simulations could help to rule out the emergence of discontinuities in the optimal path.
\par 
Although we did not address the case of $S_1$, the method should carry over straightforwardly.

\appendix 
\section{Filling the jumps}
Let us fix $\ep$ and $c$ and we define $En  : \mathbb{R} \setminus \{c\} \to \mathbb{R}$ and $Sq :\mathbb{R} \to \mathbb{R}$ as 
$$ En(x)= \begin{cases} x \qquad & \text{ if }x<c \\x+ \ep & \text{ if }x>c \end{cases} \qquad \qquad Sq(x)= \begin{cases} x \qquad & \text{ if }x<c \\c & \text{ if }c \leq x \leq c+\ep \\x- \ep & \text{ if }x>c+\ep. \end{cases}$$
In this way $Sq \circ En = id$.

\begin{definition}\label{def:FG} Let us consider a set $X=\{ x_i \}_{i \in I} \subset (a,b)$, where $I$ is either finite or countable. Let us consider also $ \ep_i \in (0,1) $ such that $\sum_{ i \in I} \ep_i = \ep$. Then we define jump function $F:[a, b] \to [a, b+ \ep]$ and the stairs function $G: [a,b+\ep] \to [a,b]$ in the following way:
$$  F(x) = x + \sum_{ x_i < x } \ep_i, \qquad G(y) = \inf\{ x \in [a,b]\; : \; F(x) \geq y\}.  $$
\end{definition}

This definition is suited for opening gaps in correspondence of some points, translating the behavior of the function. While it is pretty clear what happens if $I$ is finite we need the following lemma in order to use some properties of $F$ and $G$.

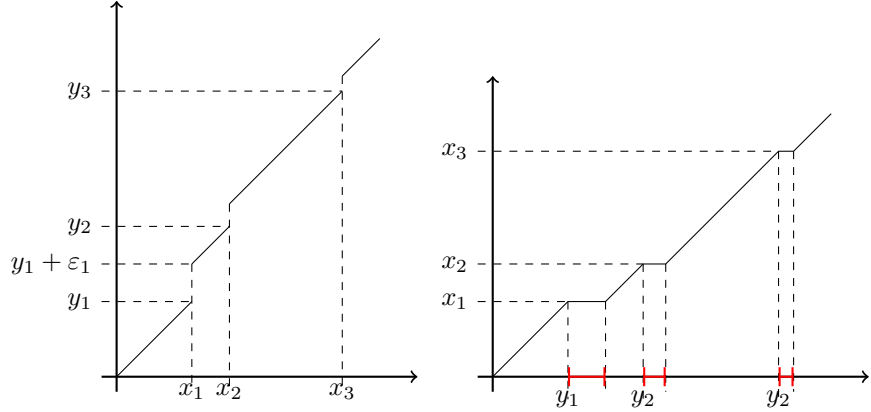
\begin{figure}
\centering

\begin{tikzpicture}
\draw[thick,->] (-0.2,0)  -- (4,0);
\draw[thick,->] (0,-0.2) -- (0,5);
\draw[-] (0,0) -- (1,1);
\draw[dashed] (1,1.5) -- (1,-0.2) node {$x_1$};
\draw[dashed] (-0.2,1) node[anchor=east] {$y_1$} -- (1,1);
\draw[dashed] (-0.2,1.5) node[anchor=east] {$y_1+ \ep_1$} -- (1,1.5);
\draw[-] (1,1.5) -- (1.5,2);
\draw[dashed] (1.5,2.3) -- (1.5,-0.2) node {$x_2$};
\draw[dashed] (-0.2,2) node[anchor=east] {$y_2$} -- (1.5,2);
\draw[-] (1.5,2.3) -- (3,3.8);
\draw[dashed] (3,4) -- (3,-0.2) node {$x_3$};
\draw[dashed] (-0.2,3.8) node[anchor=east] {$y_3$} -- (3,3.8);
\draw[-] (3,4) -- (3.5,4.5);

\draw[thick,->] (4.8,0)  -- (10,0);
\draw[thick,->] (5,-0.2) -- (5,4);
\draw[-] (5,0) -- (6,1) -- (6.5,1) -- (7,1.5) -- (7.3,1.5) -- (8.8,3) -- (9,3) --(9.5,3.5);
\draw[dashed] (4.8,1) node[anchor=east] {$x_1$} -- (6,1) ;
\draw[dashed] (4.8,1.5) node[anchor=east] {$x_2$} -- (7,1.5) ;
\draw[dashed] (4.8,3) node[anchor=east] {$x_3$} -- (8.8,3) ;
\draw[dashed] (6,1) -- (6,-0.3) node {$y_1$};
\draw[dashed] (7,1.5) -- (7,-0.3) node {$y_2$};
\draw[dashed] (8.8,3) -- (8.8,-0.3) node {$y_2$};
\draw[dashed] (6.5,1) -- (6.5,-0.2);
\draw[dashed] (7.3,1.5) -- (7.3,-0.2);
\draw[dashed] (9,3) -- (9,-0.2);
\draw[thick, red, |-|] (6,0) -- (6.5,0) ;
\draw[thick, red, |-|] (7,0) -- (7.3,0) ;
\draw[thick, red, |-|]  (8.8,0) -- (9,0);

\end{tikzpicture}
\caption{On the left the graph of $F$ and on the right the graph for $G$. In red is indicated the set $\mathcal{N}$.} \label{fig:FG}
\end{figure}

\begin{lemma}\label{lem:FG} Let $X$, $F$ and $G$ as in the previous definition. Then, letting $F(x_i)=y_i$ and $\mathcal{N}=\bigcup_{i \in I} (y_i, y_i+\ep_i)$ we have
\begin{itemize}
\item[(i)] $G(F(x))=x$;
\item[(ii)] $t \leq G(y)$ iff $F(t) \leq y$;
\item[(iii)] for every continuous $\phi \in BV(a,b)$ we have $F_{\sharp} D\phi = D (\phi \circ G)$;
\item[(iv)] for every $f \in L^1(a,b)$ we have $F_{\sharp} ( f \mathcal{L}|_{(a,b)}) = f \circ G \cdot  \mathcal{L}|_{(a,b+\ep) \setminus \mathcal{N}}$
\end{itemize}
\end{lemma}

\begin{proof}

Since $F$ is strictly increasing we have 
$$G ( F(x')) = \inf\{ x \in [a,b]\; : \; F(x) \geq F(x')\} =  \{\inf\{ x \in [a,b]\; : \; x \geq x'\} =x'. $$
We can prove (ii) again thanks again to the fact that $F$ is increasing and right continuous we have
$$t \leq G(y) \Leftrightarrow  (F(x) \geq y  \Rightarrow t \leq x ) \Leftrightarrow F(t) \leq y.$$

As for (iii) we have that for every $y \in (a,b+ \ep]$ we have 
$$ D(u \circ  G ) ( [a,y]) = u (G(y)) - u(G(a)) = D u ( [a, G(y)] ) = Du ( F^{-1} [ a,y]) ,$$
where in the last passage we used (ii).
We first prove (iv) for $f \equiv 1$. In order to prove this let us first notice that $G$ is $1$-Lipschitz, thanks to the fact that $|F(x) -F(x')|\geq |x-x'|$, and so $0 \leq DG \leq \mathcal{L}|_{[a, b+ \ep]}$; moreover clearly $G' \equiv 0$ on $(y_i , y_i + \ep_i)$. This, together with (ii) when $u(x)=x$, lead to
$$F_{\sharp} \mathcal{L}|_{[a,b]} = F_{\sharp} Du = D (u \circ G) = DG \leq \mathcal{L}|_{[a, b+ \ep] \setminus \mathcal{N} }.$$
Moreover it is obvious that $\mathcal{L}(\mathcal{N})= \ep$ and so $\mathcal{L}|_{[a, b+ \ep] \setminus \mathcal{N} } ( [a, b+ \ep] ) = b-a = Du [a,b]$. This proves that $F_{\sharp} \mathcal{L}|_{[a,b]} = \mathcal{L}|_{[a, b+ \ep] \setminus \mathcal{N}}$.

Now, in order to prove (iv) we use that since $G$ is a left inverse of $F$ we have $F_{\sharp} (f \mu) = (f \circ G) \cdot  F_{\sharp} \mu$, with $\mu= \mathcal{L}|_{[a,b]} $.
\end{proof}

\begin{lemma}[Jumps to ramps]
Let $\varphi:[a,b] \to [0,1]$ be a monotone Lagrangian trajectory such that for a.e. $x \in [a,b]$ and for $x \in \{a,b\}$ we have $\varphi(\cdot, x) \in H^1([0,1])$ with $\partial_t \varphi = v_t \circ \varphi$ for some velocity field $v_t $ such that $ \iint v_t^2+ v_t'^2  \, dx \, dt < \infty$. Let us suppose that there is an at most countable set $X = \{ x_i \}_{i \in \mathbb{I}}$ such that for every $t$ we have $J(\varphi) \subseteq X$.
Then there exists $\varphi_{\ep}:[a,b+\ep] \to [0,1]$ a monotone Lagrangian trajectory for $\tilde{v}_t$ such that $\varphi_{\ep}(t, \cdot)$ is continuous for every $t$; moreover we have $\varphi_{\ep}(t,F(x))=\varphi(t,x)$ and $D\varphi_{\ep} = F_{\sharp} D \varphi + \mu$ where $\mu \ll \mathcal{L}$ and $F$ as in Definition \ref{def:FG}.
\end{lemma}

\begin{proof} For every $x_i \in X$ we define $\varphi^{\pm}_i(t) =\lim_{x \to x_i^{\pm}} \varphi(x,t)$.  We want to prove that
\begin{itemize}
\item $\forall i \in I$ we have $\varphi^{\pm}_i \in W^{1,1}$;
\item letting $v^{\pm}_i (t)= \partial_t \varphi_i^{\pm} (t)$ we have
 $$\int_0^1 E_{sh} (v_i^-(t),v_i^+(t),\varphi^+(t),\varphi^-(t)) < \infty.$$
 \end{itemize}

This may let us use Propostition \ref{prop:gfilling} which gives us functions $\varphi_i:[0,1]\times[0,1] \to [0,1]$. These will let us construct our $\varphi_\ep$ in the following way: let us consider $\varphi^c$ the continuous part of $\varphi$. Then let us consider for every $i$ the functions
$$\varphi^{\ep}_i (t,y) = \begin{cases} 0 \qquad & \text{ if }y <y_i \\ \varphi_i (t,\frac{y-y_i}{\ep_i} ) - \varphi_i^- (t) & \text{ if }y  \in [y_i, y_i+\ep_i] \\ \varphi_i^+ (t)- \varphi_i^- (t) & \text{ if }y > y_i. \end{cases}$$ 
We then sum them up to get $\varphi_{\ep}(t,y) = \varphi^c ( t,G(y))  + \sum_{i \in I}  \varphi^{\ep}_i (t,y)$.
\end{proof}

$$ E_{sh} (v^-,v^+,a,b) = \inf \left\{ \int_a^b (v^2+ v'^2)\, dx \; : \; v(a) =v^- \: , \: v(b)=v^+  \right\}\,,$$

$$v (x) = v^+ \frac{\sinh (x-a)}{\sinh(b-a)}+ v^- \frac{\sinh (b-x)}{\sinh(b-a)}\,.$$

\begin{proposition}[Filling]\label{prop:filling}
Let $\varphi_0(t), \varphi_1(t)$ be two curves such that $\partial_t \varphi_i(t) = v_i (t)  $ for some $v_t$ such that $\int_0^1 E_{sh}(v_0(t),v_1(t), \varphi_0(t), \varphi_1(t)) < \infty$. Suppose moreover that $\varphi_0(1)=\varphi_1(1)$ and that $\varphi_0(t) <\varphi_1(t)$ for $0\leq t <1$. Then there exists $\varphi (x,t)$ monotone Lagrangian solution with velocity $v_t$, the solution to the minimization problem $E_{sh}$, such that $\varphi(x,0)=x\varphi_0(0) + (1-x)\varphi_1(0)$,  $\varphi(i,t)=\varphi_i(t)$ and $D_x\varphi(x,t) \ll \mathcal{L}$ for every $t<1$. 
\end{proposition}

\begin{proof} Let $\Omega= \{ (x,t) \; : \;  0 \leq t <1, \; \varphi_0(t) < x < \varphi_1 (t) \}$. In this set we have that $v_t$ is locally smooth in $x$; moreover we have
 $$Lip_x (v_t) \leq (|v_0(t)|+|v_1(t)|)/\tanh(\varphi_0(t)-\varphi_1(t)),$$
which implies that $\int_0^s Lip_x (v_t) \, d t < \infty$ for all $0 \leq s < 1$. This already implies that the Cauchy problem is well posed in $\Omega$ and in particular, fixing $\varphi_0(0) < f(0) < \varphi_1 (0)$, a unique solution to $\partial_t f(t) = v_t (f(t))$ exists up until it hits the boundary of $\Omega$. Suppose that $(s, f(s))$ is on the boundary of $\Omega$: we want to prove that $s=1$ (and in particular also $f(s)=\varphi_1(1)=\varphi_0(1)$). In fact if $s <1$ without loss of generality we can assume $f (s)=\varphi_1(s)$; but since $\varphi_1$ satisfies the same Cauchy problem we have $ \partial_t (\varphi_1(t) - f(t)) \geq - Lip_x (v_t) (\varphi_1(t) - f(t) )  $ and so by Gronwall
$$|\varphi_1(s)- f(s)| \geq | \varphi_1(0) -f(0)| e^{- \int_0^s Lip_x (v_t) \, d t },$$
which is a contradiction since we supposed $f(0) \neq \varphi_1(0)$.

As for the second point we have that for $y<y'$, $\varphi (y', 0) - \varphi(y,0)= (y'-y) \cdot (\varphi_1(0) - \varphi_0(0))$; a similar Gronwall argument as before shows that for every $s<1$ we have
$$ \varphi (y',s) - \varphi(y,s) \leq (y'-y) \cdot (\varphi_1(0) -\varphi_0(0))e^{\int_0^s Lip_x (v_t) \, d t }.$$
In particular we have that $\varphi(x,s)$ is Lipschitz in $x$ for every $s<1$. 
\end{proof}

\begin{proposition}[General Filling]\label{prop:gfilling}
Let $\varphi_0(t), \varphi_1(t)$ be two curves such that $\partial_t \varphi_i(t) = v_i (t)  $ for some integralble $v_i(t)$ such that $\int_0^1 E_{sh}(v_0(t),v_1(t), \varphi_0(t), \varphi_1(t)) < \infty$. Suppose moreover that $\varphi_0(t) \leq \varphi_1(t)$ for every $t \in [0,1]$. Then there exists $\varphi (x,t)$ monotone Lagrangian solution with velocity $v_t$, the solution to the minimization problem $E_{sh}$, such that  $\varphi(i,t)=\varphi_i(t)$ and $D_x\varphi(x,t) \ll \mathcal{L}$ for every $t \in [0,1]$. 
\end{proposition}

\begin{proof} The curves $\varphi_0(t), \varphi_1(t)$ are continuous and so the set $\{ t \in [0,1] \; \varphi_0(t) < \varphi_1(t) \} =\Omega$ is open in $[0,1]$. In particular it is a countable union of disjoint intervals $I_i =(a_i,b_i)$. On every half interval $[\frac{a_i+b_i}2, b_i)$ and $(a_i, \frac{a_i+b_i}2]$ we can apply Proposition \ref{prop:filling} in order to define $\varphi$ on $[a_i,b_i] \times [0,1]$ (notice that the construction in $\frac{a_i+b_i}2$ is the same from both sides). Then it is sufficient to define $\varphi(x,t)=\varphi_0(t)$ on $[0,1] \setminus \Omega$.

\end{proof}

%

\bibliographystyle{plain}      
\bibliography{articles,references,SecOrdLandBig,MesPapiers}   

\end{document}